\documentclass [11pt,reqno]{amsart}

\usepackage{graphicx}
\usepackage{indentfirst}
\usepackage[mathscr]{eucal}
\usepackage[all]{xy}
\usepackage{amsmath,amsxtra,amssymb,latexsym,amscd,amsthm,wasysym, amsfonts,color,esint,mathtools,geometry}
\usepackage{fancyhdr}
\usepackage{dsfont}
\usepackage{comment}
\usepackage{enumitem}
\usepackage{inputenc}

\def\PSH{\mathcal{PSH}}
\def\SH{\mathcal{SH}}

\def\ddc{dd^c}

\def\osc{\mathrm{Osc}}

\numberwithin{equation}{section}
\def\1{\mathds{1}}

\usepackage{cite}
\usepackage{hyperref}
\hypersetup{
    unicode=false,        
    pdftoolbar=true,      
    pdfmenubar=true,       
    pdffitwindow=false,     
    pdfstartview={FitH},    
    pdftitle={},    
    pdfauthor={},     
    colorlinks=true,       
   linkcolor=blue,          
    citecolor=blue,        
    filecolor=black,      
    urlcolor=black} 

\title[Continuity of solutions to complex Hessian equations]{Continuity of solutions to complex Hessian equations on compact Hermitian manifolds}
\author{Yuetong Fang}

\address{Université d'Angers, CNRS, LAREMA, SFR MATHSTIC, F-49000 Angers, France}
\email{yuetong.fang@univ-angers.fr}

\begin{document}
\begin{abstract}
    Let $(X,\omega)$ be a compact Hermitian manifold of dimension $n$. We derive an $L^\infty$-estimate for bounded solutions to the complex $m$-th Hessian equations on $X$, assuming a positive right-hand side in the Orlicz space $L^{\frac{n}{m}}(\log L)^n(h\circ\log \circ \log L)^n$, where the associated weight satisfies Ko{\l}odziej's Condition. Building upon this estimate, we then establish the existence of continuous solutions to the complex Hessian equation under the prescribed assumptions.
\end{abstract}
\maketitle

%mathbb letters
\newcommand{\A}{\mathbb{A}}
\newcommand{\B}{\mathbb{B}}
\newcommand{\C}{\mathbb{C}}
\newcommand{\G}{\mathbb{G}}

\newcommand{\Gm}{\mathbb{G}_\mathrm{m}}

\newcommand{\N}{\mathbb{N}}
\renewcommand{\P}{\mathbb{P}}
 \newcommand{\Q}{\mathbb{Q}}
 \newcommand{\R}{\mathbb{R}}
 \newcommand{\Z}{\mathbb{Z}}
\newcommand{\T}{\mathbb{T}}

%accent
\newcommand{\poll}{ł}
\newcommand{\swea}{å}

\newtheorem{theorem}{Theorem}[section]
\newtheorem{thm}[theorem]{Theorem}
\newtheorem{lemma}[theorem]{Lemma}
\newtheorem{lem}[theorem]{Lemma}
\newtheorem{prop}[theorem]{Proposition}
\newtheorem{coro}[theorem]{Corollary}
\newtheorem{corollary}[theorem]{Corollary}
\newtheorem*{Theorem A}{Theorem A}
\newtheorem*{Theorem B}{Theorem B}
\newtheorem*{Lemma A}{Lemma A}
\newtheorem*{Theorem C}{Theorem C}

\theoremstyle{definition}
\newtheorem{definition}[theorem]{Definition}
\newtheorem{defi}[theorem]{Definition}
\newtheorem{example}[theorem]{Example}
\newtheorem{exa}[theorem]{Example}
\newtheorem{claim}[theorem]{Claim}
\newtheorem{remark}[theorem]{Remark}
\newtheorem*{ackn}{Acknowledgment}

\section*{Introduction}
Let $(X,\omega)$ be a compact complex Hermitian manifold of dimension $n$ equipped with a Hermitian metric $\omega$. Fix $0< m \le n$ and let $dV_X$ be a smooth volume form. The problem of interest in this paper is to find a constant $c>0$ and continuous $(\omega,m)$-subharmonic solutions to the following complex $m$-Hessian equation
\begin{equation} \tag{Hessian} \label{Equation: Hessian}
     (\omega+d d^c u)^m \wedge \omega^{n-m} = c e^{\lambda u} f dV_X, \lambda\geq 0,
\end{equation}
under the assumption that the densities $f^{n/m}$ belong to an Orlicz space $L^{\chi}$ with the weight $\chi$ satisfying Ko\l odziej's condition, which will be explained later.

When $m=1$, this simplifies to the Poisson equation, while when $m=n$, it reduces to the Monge--Ampère equation. Yau’s solution \cite{Yau78} of the Calabi Conjecture has motivated many subsequent studies on Monge--Ampère equations over the past fifty years. As presented by Błocki \cite{Błocki2012Lecture}, using the simplifications due to Kazdan and to Aubin--Bourguignon, Yau's a priori estimate can be extended to cases where the densities $f$ belong to $L^p$ for $p>n$. One of the breakthrough was achieved by Ko\l odziej \cite{Kolodziej1998}, who applied pluripotential theory to obtain an $L^{\infty}$-estimate for Monge--Amp\`ere equations when $f$ belongs to some Orlicz space, with the associated weight specifically satisfying \hyperref[Condition K]{Condition (K)}. As remarked by Ko\l odziej \cite{Kolodziej1998}, such an a priori estimate is "almost sharp".

This integrability condition on densities $f$ can be extended in the context of complex Hessian equations. We say that a convex weight $\chi$ satisfies Ko\l odziej's condition (\hyperref[Condition K]{Condition (K)}) if it is increasing and \[\chi(t) \underset{t \to +\infty}{\sim}  t (\log (t+1))^n \left(h \circ \log \circ \log (t+3)\right)^n,\] where the function
$h:\R^+ \rightarrow \R^+$ satisfying  $\int^{+\infty} h^{-1}(t)dt <+\infty.$
In this paper, we focus on the case where $f^{n/m} \in L^{\chi}$.
We refer the reader to Section~\ref{Sec:Preliminaries} for more detailed discussions on Orlicz spaces and \hyperref[Condition K]{Condition (K)}.

The study of weak solutions to complex Monge--Ampère equations plays an important role in the analysis of Ricci flows (see \cite{Tosatti_2015}). On compact Kähler manifolds, one can study the modulus of continuity for solutions to complex Monge–Ampère equations. Hölder regularity for such equations with densities in $L^p$ was established in \cite{DDG+14}, and has been further generalized to the case where the densities belong to certain Orlicz spaces, in \cite{guo2021moduluscontinuitysolutionscomplex} via a PDE approach, and in \cite{GGZ25} using pluripotential theory.

To prove the existence of continuous solution to the complex Monge--Amp\`ere equation, a key ingredient is the $L^{\infty}$-estimate. Various techniques have been developed to obtain $L^{\infty}$-estimate for Monge--Amp\`ere equations on compact manifolds, including Yau's \cite{Yau78} original proof, which employs Moser's iteration scheme. Later, Kołodziej \cite{Kolodziej1998} developed $L^{\infty}$-estimates using properties of capacity and established an optimal integrability condition for the density $f$ (see also~\cite{EGZ08,DP2010} for related generalizations). Inspired by the work of Chen--Cheng \cite{chencheng2021cscK1} on constant scalar curvature Kähler metrics, a very different approach was introduced by Guo--Phong--Tong \cite{GPT23}, by comparing the given equation to an auxiliary complex Monge--Ampère equation. This PDE technique has since been generalized in \cite{GP24, qiao2024sharpmathrmlinftyestimatesfully}. An efficient approach was developed by Guedj--Lu \cite{GL21, GL21a}, using properties of envelopes, which can be adapted to the case where $f \in L \log^p L$ with $ p>n$, as noted in \cite[Section 2]{GL21a}. More recently, they generalized this result to Orlicz spaces satisfying \hyperref[Condition K]{Condition (K)} on compact Kähler manifolds in \cite{gl2025}, using a new approach that reduces Ko{\l}odziej's criterion to Yau's theorem.

 On the other hand, complex $m$-Hessian equations are natural generalizations of Monge--Amp\`ere equations. B\l ocki \cite{Blocki05} established potential theory for $m$-Hessian equations in domains of $\C^n$. Later, Dinew--Ko\l odziej \cite{DK14} obtained a priori estimates and proved a stability result for weak solutions of the complex Hessian equation in domains of $\C^n$ and compact K\"ahler manifolds. The modulus of continuity of solutions to complex Hessian equations has been studied in \cite{Nguyen_2014,Cha16} with densities $f\in L^p,p>n/m$, and has been extended recently in \cite{AC24} to the cases where $f \in L^{n/m}\log^p(L),p>2n$. On compact Hermitian manifolds, Kołodziej and Nguyen proved the existence of weak solutions with a nonnegative right-hand side belonging to $L^p$, $p>n/m$, in \cite{Ko_odziej_2016}, while the corresponding result for Monge–Ampère equations was established in \cite{KN15weaksol}.

The purpose of this paper is to extend the above results to prove the continuity of bounded solutions to the complex Hessian equation. We begin by establishing a subsolution lemma for the Monge--Amp\`ere equation, generalizing Lemma 2.1 of \cite{GL21} to the case where $f \in L^{\chi}$, with $\chi$ satisfying \hyperref[Condition K]{Condition (K)}. The lemma is stated as follows:

\begin{Lemma A}[Subsolution Lemma] \hypertarget{LemA}{}
Assume $\chi$ satisfies \hyperref[Condition K]{Condition (K)}. Then 
there exist a uniform constant $s=s(\chi)>0$ such that for all $0\leq g\in L^{\chi}$, we can find $\psi\in \PSH(X,\omega)$, $-1\leq \psi \leq 0$, satisfying
  \[
  (\omega +\ddc \psi)^n \geq  s(\chi) \frac{g}{\|g\|_{\chi}}  dV_X.
  \]  
\end{Lemma A}

We obtain the Subsolution Lemma by gluing the solutions to the Dirichlet problem in local context:
\begin{equation} \label{Prob:Dirch}
    \begin{aligned}  
   (\ddc  u)^n &= \frac{f }{ \| f\|_{\chi} } dV_X \text{, in }\Omega;\\
   u &= 0 \text{, on }\partial \Omega.
   \end{aligned}
\end{equation}

Compared with the case where the densities $f \in L^p$, working with Orlicz norms is more involved, as it requires controlling the Orlicz norm of the solution $u$ associated to the convex conjugate of $\chi$. To simplify the argument, we adapt an approach from the recent article by Guedj–Lu \cite{gl2025} to derive a local $L^{\infty}$-estimate for problem~\ref{Prob:Dirch}.

As a consequence of \hyperlink{LemA}{Lemma A}, we obtain a uniform estimate for the complex Hessian equation on a Hermitian manifold, where the densities $f $ satisfy Ko{\l}odziej's optimal condition.

\begin{Theorem B} [$L^{\infty}$-estimate] \hypertarget{ThmB}{}
Let $\varphi$ be a bounded $(\omega,m)$-subharmonic function satisfying
\begin{equation} \tag{Hess} \label{ean:Hess}
(\omega+ \ddc \varphi)^m \wedge \omega^{n-m} \le f d V_X,
\end{equation}
where $0\le f^{n/m} \in L^{\chi}$ with the weight $\chi$ satisfying \hyperref[Condition K]{Condition (K)}. Then
\[
\osc_X (\varphi) \le C,
\]
where $C$ depends on $\omega, n, m$ and $\left\| f^{\frac{n}{m} }\right\|_{\chi}$.
\end{Theorem B}

By generalizing an idea in \cite{guedj2023degeneratecomplexhessianequations}, we demonstrate that the $L^{\infty}$-estimate for complex Hessian equations can be reduced to the corresponding estimate for an associated Monge--Amp\`ere equation:
\begin{equation}
    (\omega + \ddc \varphi)^n \le f dV_X,
\end{equation}
where $f \in L^{\chi}$, and $\chi$ satisfies \hyperref[Condition K]{Condition (K)}. To obtain the $L^{\infty}$-estimate for the Monge--Ampère equation, our main tool is the domination principle.
The key novelty of our approach lies in constructing an appropriate subsolution.  Our approach avoids potential computational difficulties associated with Orlicz norms and presents a different method compared to Ko\l odziej's original proof. 

Once the $L^{\infty}$-estimate is obtained, we can derive the existence and regularity of bounded solutions. This leads us to the second main result of the study:

\begin{Theorem C} \hypertarget{ThmC}{}
Let $0 \le f^{n/m} \in L^{\chi}$, with the weight function $\chi$ satisfying \hyperref[Condition K]{Condition (K)}.
Then there exist a constant $c >0$ and a solution $u \in \SH_m(X, \omega)\cap \mathcal{C}^0(X)$ satisfying the complex Hessian equation
\[ (\omega+d d^c u)^m \wedge \omega^{n-m} = c e^{\lambda u} f dV_X,\; \lambda\geq 0.
\]
Moreover, all bounded solutions to the above complex Hessian equation are continuous.
\end{Theorem C}

To prove the first statement of \hyperlink{ThmC}{Theorem C}, an essential component is to establish a stability theorem (Theorem~\ref{Thm:stab}). For $f \in L^p$, with $p>n/m$, such a stability result was proven in \cite{Ko_odziej_2016}. In the case where $f\in L^{\chi}$ and the weight function $\chi $ satisfies \hyperref[Condition K]{Condition (K)}, the continuity of solutions to Monge–Ampère equations was studied in \cite{KN2021continuous} by adapting the method in \cite[Theorem 3.1]{KN2019Stability}. We point out that our approach to the stability estimate is fundamentally different and remains of interest in the case where the densities $f\in L^p$, as discussed in Remark~\ref{rem: Lpstability}. Our main tool will be the Domination Principle and the subsolution  (\hyperlink{LemA}{Lemma A}). Having established the stability estimate (Theorem~\ref{Thm:stab}), we approximate the density $f$ by the sequence of bounded functions $\min(f,j)$. We show that the corresponding solutions $u_j $ converge uniformly to an $(\omega,m)$-subharmonic function $u$ solving the equation \eqref{Equation: Hessian}, thereby obtaining a continuous solution. The continuity of all bounded solutions follows from the Domination Principle (Proposition~\ref{prop: DP}).

\subsection*{Organization of the paper}The paper is organized as follows. Section~\ref{Sec:Preliminaries} gives preliminary material on $(\omega,m)$-subharmonic functions, as well as key concepts of Orlicz spaces. A local $L^{\infty}$-estimate is established, and Lemma A and \hyperlink{ThmB}{Theorem B} are proved in Section~\ref{Sec:unifest}. In section~\ref{sec:weaksol}, we establish a stability estimate and subsequently investigate the continuity of bounded solutions, leading to the proof of \hyperlink{ThmC}{Theorem C}.

\begin{ackn}
  The author thanks her supervisor, Hoang-Chinh Lu, for suggesting the problem, valuable discussions during the preparation of the work, and useful comments that improved the presentation. This research is part of a PhD program funded by the PhD scholarship CSC-202308070110, and is partially supported by the projects Centre Henri Lebesgue ANR-11-LABX-0020-01 and the KRIS project of Fondation Charles Defforey.   
\end{ackn}

\section{Preliminaries} \label{Sec:Preliminaries}
In the whole paper, we let $(X,\omega)$ denote a compact Hermitian manifold of complex dimension $n\in \mathbb{N}^{*}$, equipped with a Hermitian form $\omega$. We denote by $dV_X$ a smooth volume form on $X$. We use differential operators $d = \partial + \bar{\partial}$, and $d^c = i (\bar{\partial}- \partial)$, so that $\ddc = 2i \partial \bar{\partial}$.

\subsection{\texorpdfstring{$(\omega,m)$}{}-subharmonic functions}

\subsubsection{Definition}
We recall the definition of $(\omega,m)$-subharmonic functions following \cite{GuNguyen2018} and briefly review some results related to both $\omega$-plurisubharmonic and $(\omega,m)$-subharmonic functions that will be used later.

Fix an integer $1 \le m \le n$. Fix $\Omega $ an open set in $\C^n$. Given a hermitian metric $\omega$ on $\Omega$, a real $(1,1)$-form $\alpha$ on $X$ is called $m$-positive with respect to $\omega$ if at all points in $X$,
\[
\alpha^{k} \wedge \omega^{n-k} \ge 0, \forall k =1, \cdots, m.
\]
A $\mathcal{C}^2(\Omega)$ function $u : \Omega \rightarrow  \R$ is called 
harmonic with respect to $\omega$ if $\ddc u \wedge \omega^{n-1}=0$ at all points in $\Omega$.

\begin{definition}
A function $u:\Omega \rightarrow \{ -\infty\} \cup \R$ is subharmonic with respect to $\omega$ if:
\begin{enumerate}[label=(\alph*)]
\item $u$ is upper semicontinuous and $u \in L^1_{loc}(\Omega)$;
\item for every relatively compact open set $D \Subset \Omega$ and every function $h\in \mathcal{C}^0(D)$ that is harmonic with respect to $\omega$ on $D$ the following implication holds:
\[
u \le h \text{ on }\partial D \implies u \le h \text{ in }D.
\]
\end{enumerate}
\end{definition}

\begin{definition}
A function $\varphi : \Omega \rightarrow  \{ -\infty\} \cup \R$ is quasi-subharmonic with respect to $\omega$ if locally $\varphi = u + \rho$, where $u$ is subharmonic with respect to $\omega $ and $\rho$ is smooth.

A function $\varphi$ is $\alpha$-subharmonic with respect to $\omega$ if $\varphi$ is quasi-subharmonic with respect to $\omega$ and $(\alpha + \ddc \varphi) \wedge \omega^{n-1}\ge 0$ in the sense of distributions.
\end{definition}
The positive cone $\Gamma_m(\omega)$ associated with the metric $\omega$ is defined by 
\[
\left\{  \gamma \text{ real }(1,1) \text{-form}: \gamma^k\wedge \omega^{n-k}>0, k=1,\cdots,m \right\}.
\]
It follows from Gårding's inequality\cite{gar59} that if $\gamma_0, \gamma_1, \cdots, \gamma_{m-1}\in \Gamma_{m}(\omega)$, then 
\[
\gamma_0 \wedge\gamma_1\wedge \cdots \wedge\gamma_{m-1} \wedge\omega^{n-m} >0. 
\]
One can write $\widetilde{\omega}= \gamma_1\wedge \cdots \wedge\gamma_{m-1} \wedge\omega$, and it is a strictly positive $(n-1,n-1)$-form on $\Omega$.
\begin{definition}
A function $\varphi: \Omega \rightarrow [ -\infty, + \infty)$ is called $(\alpha,m)$-subharmonic with respect to $\omega$ if $\varphi$ is $\alpha$-subharmonic with respect to $\widetilde{\omega}$ in $\Omega$ for all 
$\widetilde{\omega}$ of the form $\widetilde{\omega}^{n-1}= \gamma_1\wedge \cdots \wedge\gamma_{m-1} \wedge\omega$, where $\gamma_1, \cdots,\gamma_{m-1} \in \Gamma_m(\omega)$.

A function $u : X \rightarrow \mathbb{R} \cup \{- \infty \} $ is called $(\alpha,m)$-subharmonic on $X$ if $u$ is $(\alpha,m)$-subharmonic on each local chart $U$ of $X$.
\end{definition}
The set of all locally integrable functions on $U$ which are $(\alpha,m)$-subharmonic with respect to $\omega$ in $U$ is denoted by $\SH_{\omega,m}(U,\alpha)$. In many contexts, such as \cite{guedj2023degeneratecomplexhessianequations}, the forms $\omega$ and $\alpha$ may differ.

However, in this paper, we focus on the set $\SH_{\omega,m}(U,\omega)$. To simplify the notations, we denote by $\SH_m(U,\omega)$ the set of all $(\omega,m)$-subharmonic functions with respect to $\omega$ in $U$. The set of all $\omega$-plurisubharmonic functions on $U$ is denoted by $\PSH(U, \omega)=\SH_n(U,\omega)$.

For notational convenience, we denote $\omega_{u} \coloneqq \omega + \ddc u$.

\begin{remark}
  By Gårding's inequality\cite{gar59}, if $u \in \mathcal{C}^2(X)$, then $u$ is $(\omega,m)$-subharmonic with respect to $\omega$ on $X$ if and only if the associated form $\omega _u$ belongs to the closure of $\Gamma_m(\omega) $.
\end{remark}
It is well known that when $n=m$, the set $\PSH(X,\omega)$ is a closed subset of $L^r(X)$, for all $r \ge 1$. For $(\omega,m)$-subharmonic functions, the following $L^1$-compactness result was proved in \cite[lemma 3.3]{Ko_odziej_2016}.

\begin{lem} \label{lem:comp}
    Let $u\in \SH_m(X,\omega)$ be normalized by $\sup_X u =0$. Then there exists a uniform constant $A>0 $ depending only on $X, \omega$ such that
    \[
    \int_X |u|\omega^n \le A.
    \]
\end{lem}

\subsubsection{Domination principle}
Analogous to the Monge--Ampère operators, the complex Hessian operators also satisfy a domination principle. The domination principle for Hessian operator associated with continuous $(\omega,m)$-subharmonic functions is established by Guedj--Lu in \cite[Section 2]{guedj2023degeneratecomplexhessianequations}. Following the demonstration that the Hessian operators for bounded functions are well-defined, we show the domination principle associated with bounded $(\omega,m)$-sh functions here.
\begin{lem} [Maximum principle] \label{Lem:maxp}
Assume that $u$ and $v$ are bounded $(\omega,m)$-subharmonic functions. Then
 \[
 (\omega + \ddc (\max(u,v)))^m \wedge \omega^{n-m} \ge \1_{\{ u >v \} } (\omega+\ddc u)^m \wedge \omega^{n-m} + \1_{\{ v\ge u \} } (\omega+\ddc v)^m \wedge \omega^{n-m}.
 \]
 In particular, if $u \le v$, then,
 \[
 \1_{\{ v= u \} } (\omega+\ddc u)^m \wedge \omega^{n-m} \le \1_{\{ v= u \} } (\omega+\ddc v)^m \wedge \omega^{n-m}
 \]
\end{lem}
In the case where $u$ and $v$ are continuous, the above lemma is proven in \cite[Lemma 1.10]{guedj2023degeneratecomplexhessianequations}. When $u$ is bounded, thanks to the continuity of complex Hessian operators (\cite{kolodziej2023complexhessianmeasuresrespect}), we can approximate $u$ by a sequence of decreasing smooth $(\omega,m)$-subharmonic functions.

The comparison principle for $u,v \in \SH_m(\Omega) \cap L^{\infty}(\Omega)$ is introduced in \cite[Section 6]{kolodziej2023complexhessianmeasuresrespect}.
\begin{prop} [Comparison principle] \label{prop:CP}
    Let $u,v$ be bounded $m$-subharmonic functions with respect to $\omega$ and such that $\liminf_{z \rightarrow \partial\Omega}u(z)-v(z) \ge0$. Assume \[
     (\ddc v)^m \wedge \omega^{n-m}\ge(\ddc u)^m \wedge \omega^{n-m} \text{ in } \Omega.\] Then it follows that $u \ge v $ in $\Omega$.
\end{prop}
\begin{proof}
    See \cite[Corollary 6.2]{kolodziej2023complexhessianmeasuresrespect}.
\end{proof}

We will need a global version of domination principle:
\begin{prop}[Domination principle] \label{prop: DP}
Assume that $u, v \in \PSH(X, \omega) \cap L^{\infty}(X)$ satisfy
\[
1_{\{u<v\}}\left(\omega+d d^c u\right)^m \wedge \omega^{n-m} \leq c 1_{\{u<v\}}\left(\omega+d d^c v\right)^m \wedge \omega^{n-m},
\]
for some constant $0 \le c<1$. Then $u \ge v$.
\end{prop}
The case in which $u,v$ are continuous has been treated in \cite[Section 2]{guedj2023degeneratecomplexhessianequations}. In the case where $u,v \in L^{\infty}$, our proof is essentially analogous to the one developed for the Monge--Amp\`ere operators in \cite{BGL24}.
\begin{proof}
  we may assume $v \geq 1$. We fix a constant $a$ such that $c<a^m<1$, and we prove that $u \geq a v$. The result then follows by letting $a \rightarrow 1$. Assume it is not the case and let $\{x_j\}$ be a sequence converging to $x_0 \in X$ with
\[
\lim _{j \rightarrow+\infty}\left(u\left(x_j\right)-a v\left(x_j\right)\right)=m_a:=\inf _X(u-a v)<0.
\]
Let $B$ be a small neighborhood of $x_0$ and fix a smooth strictly plurisubharmonic function $\rho$ such that $\rho=0$ on $\partial B,|\rho| \leq v$, and $d d^c \rho \leq \omega$. One can take a holomorphic coordinate chart around $x_0$ which is biholomorphic to the unit ball and $\rho(x)=b\left(|x|^2-1\right)$, for a small constant $b$. Observe that $\rho<0$ in $B$. Consider now $\varphi=a v-(1-a) \rho$.
A direct computation then shows that $\omega+d d^c \varphi \geq 0$ in $B$. Also, by construction $\varphi \leq v$ and so \[\{u<\varphi\} \cap B \subset\{u<a v-(1-a) \rho\} \subset\{u<v\}.\] Hence by assumption and using the fact that $\omega+d d^c(-\rho) \geq 0$, we get
\[
\begin{aligned}
\1_{\{u<\varphi\}}\left(\omega+d d^c u\right)^m \wedge\omega^{n-m} \leq & c \1_{\{u<\varphi\}}\left(\omega+d d^c v\right)^m \wedge \omega^{n-m}\\
\leq& a^m \1_{\{u<\varphi\}}\left(\omega+d d^c v\right)^m\wedge \omega^{n-m}\\ \leq & \1_{\{u<\varphi\}}\left(\omega+d d^c \varphi\right)^m \wedge \omega^{n-m}.
\end{aligned}
\]
From this and the maximum principle (Lemma~\ref{Lem:maxp}) we infer
\[
\left(\omega+d d^c u\right)^m \wedge \omega^{n-m} \leq\left(\omega+d d^c \max (u, \varphi)\right)^m \wedge \omega^{n-m} \text { in } B.
\]
It thus follows from the comparison principle (Proposition~\ref{prop:CP}) that
\[
\inf _B(u-\max (u, \varphi)) \geq \liminf _{x \rightarrow \partial B}(u-\max (u, \varphi))(x)
\]
Evaluating this inequality at $x_j$ and letting $j \rightarrow+\infty$, we arrive at
\[\lim _{j \rightarrow+\infty} \min \left(u\left(x_j\right)-a v\left(x_j\right)+(1-a) \rho\left(x_j\right), 0\right) \geq \liminf _{x \rightarrow \partial B} \min \left(u\left(x\right)-a v\left(x\right), 0\right).\]
Since by construction,
$$
\lim _{j \rightarrow+\infty} u\left(x_j\right)-a v\left(x_j\right)+(1-a) \rho\left(x_j\right) = m_a - (1-a)b,
$$
we infer that
$$
m_a - (1-a)b\ge \lim  _{x \rightarrow \partial B}\left(u\left(x\right)-a v\left(x\right)\right) \ge m_a,
$$
which leads to a contradiction.
\end{proof}

As a direct consequence of the domination principle, we have
\begin{coro}
    Fix $\lambda>0$. If $e^{-\lambda v} \omega_v^m \wedge \omega^{n-m} \ge e^{-\lambda u} \omega_u^m \wedge \omega^{n-m}$, then $u \ge v$.
\end{coro}

\begin{proof}
    Fix $C > 0$. On the set $\{ u < v-C \}$ we have \[\omega_u^m \wedge \omega^{n-m} \le e^{\lambda (u-v)} \omega_v^m \wedge \omega^{n-m} \le e^{-\lambda C} \omega_v^m \wedge \omega^{n-m} .\] It follows from Proposition~\ref{prop: DP} that $ u \ge v - C$. Since this holds for all $C >0$, we conclude that $u \ge v $.
\end{proof}
\begin{coro}
Assume $u, v$ are bounded $(\omega,m)$-subharmonic functions on $X$ such that
\[
\left(\omega+d d^c u\right)^m \wedge \omega^{n-m} \leq c\left(\omega+d d^c v\right)^m \wedge \omega^{n-m}
\]
for some positive constant $c$. Then $c \geq 1$.
\end{coro}

\subsubsection{Cegrell classes}
In Section~\ref{Sec:unifest}, we study the Monge--Ampère equation in a local context. Here, we recall some essential definitions of the Cegrell classes \cite{Cegrell1998,Cegrell04}. 

Fix a bounded hyperconvex domain $\Omega \subset \C^n$. Let $\mathcal{T}(\Omega)$ denote the set of bounded plurisubharmonic functions $u$ in $\Omega$ such that $\lim_{z \rightarrow \zeta } u(z) =0$ for every $\zeta \in \partial \Omega$, and $\int_{\Omega}(\ddc u)^n < +\infty$.

A function $u$ belongs to $\mathcal{F}(\Omega)$ iff there exists a sequence $u_j \in \mathcal{T}(\Omega)$ decreasing to $u$ in all of $\Omega$, which satisfies $\sup_{j} \int_{\Omega}(\ddc u_j)^n < +\infty$.

As demonstrated by Cegrell \cite{Cegrell1998,Cegrell04}, the complex Monge--Amp\`ere operator can be defined for functions belonging to the class $\mathcal{F}(\Omega)$ %and $\mathcal{F}^{p}(\Omega)$, 
even when these functions are unbounded. The Dirichlet problem in %classes $ \mathcal{F}^{p}(\Omega)$ and 
$\mathcal{F}(\Omega)$ was solved in \cite{Cegrell04}:

\begin{thm}\label{thm:Cegrell98}
    Assume that $\mu$ is a positive measure on $\Omega$. If $ \mu(\Omega) < +\infty$ and $\mu$ vanishes on all pluripolar sets, then there exists a unique function $u \in \mathcal{F}(\Omega)$ such that $(\ddc u)^n = \mu$.
\end{thm}
\begin{proof}
    See \cite[Lemma 5.14]{Cegrell04}.
\end{proof}

\subsection{Orlicz space}
We recall some essential facts about Orlicz spaces that will be used throughout, following \cite{Dar_survay}. Let $\mu$ be a probability measure on $X$. A weight $\chi : \R^{*} \to \R^* \cup \{+\infty\}$ is called \emph{normalized Young weight} if it is convex, lower semi-continuous, non-trivial satisfying the normalizing condition \[
\chi(0)= 0 \text{ and } 1\in\partial\chi(1).
\]
Recall that $ \partial \chi(l) \subset \mathbb{R} $ denotes the set of subgradients of  $\chi$  at  $l $, which means that $ a \in \partial \chi(t)$ if and only if  $\chi(t) + a b \leq \chi(t + b)$  for all  $b \in \R$. Its conjugate convex weight $\chi^{*}$ is the Legendre transform of $\chi$:
\[
\chi^{*}(h) = \sup_{t\in \R^*} \{ ht - \chi(t)\}.
\]
One can easily verify that the following propositions hold:
\begin{prop}  Let $\chi$ be a normalized Young weight. Then 
    \begin{enumerate}[label=(\roman*)]
        \item $\chi^*$ is also a normalized Young weight;
        \item(Fenchel--Moreau theorem). The biconjugate weight $ \chi^{**}$ is also a normalized Young weight, and $\chi^{**}= \chi$.
    \end{enumerate} 
\end{prop}
One can define a function space with respect to the Young function. Let $L^{\chi}(\mu)$ denote the space of measurable functions defined as follows:
\[
L^{\chi}(\mu) \coloneqq \left\{ f: X \rightarrow \R^{*}\cup \{+\infty\}: \exists s >0, \int_X \chi(s f) < +\infty \right\}.
\]
One may also define the norm (also known as Luxembourg norm) on $L^{\chi}(\mu)$.
\[
\| f \|_{\chi,\mu} \coloneqq \inf \left\{ r>0 : \int_X \chi \left(\frac{f}{r} \right) \le \chi(1) \right\}.
\]
We remark that $(L^{\chi}(\mu), \| \cdot\|_{\chi,\mu})$ constitutes a normed space and is complete, hence a Banach space. In the remainder of the paper, we fix $\mu = dV_X$, a smooth probability measure, and write $(L^{\chi}, \| \cdot\|_{\chi})$ for simplicity.

In \cite{Kolodziej1998}, S.Ko\l odziej developped a technique using pluripotential theory that enables one to bound the solution to the Monge--Amp\`ere equation by the Orlicz norm, $\| \cdot\|_{\chi}$, of the right-hand side, associated with a weaker integrable condition. We thus introduce the \hyperref[Condition K]{Condition (K)}.
\begin{definition}[Condition (K)] \label{Condition K} 
A weight $\chi: \R^{+} \rightarrow \R^{+}$ satisfies condition (K) if it is convex increasing and  \[
\chi(t) \underset{t \to +\infty}{\sim}   t^{n/m} \left(\log(1+t)\right)^{n}  \left(h \circ \log \circ \log(1+t)\right)^{n} \text{, where } \int^{+\infty} \frac{dt}{h(t)} <+\infty.
\]
\end{definition}

\begin{example} We discuss and present a few typical examples:
    \begin{enumerate}[label=(\roman*)] 
        \item When $\chi(t)= \frac{t^p}{p}$ for $p>1$, its convex conjugate is given by $\chi^{*}(t ) = \frac{t^q}{q}$, where $q = \frac{p}{p-1}$ is the H\"older conjugate of $p$. The Orlicz norm $\|\cdot\|_{\chi}$ here is the usual $L^p$ norm.
        \item The weight $\chi$ defined in the \hyperref[Condition K]{Condition (K)} above is a Young function hence induce an Orlicz space. In particular, the weight $\chi(t)= t \log^p(1+t)$ with $p>n$ is a Young function. One can check that as $t\rightarrow +\infty$, the conjugate weight grows like $\chi^{*}(t) = t^{1-\frac{1}{p}} \exp (t^{\frac{1}{p}})$.
    \end{enumerate}
\end{example}

We will focus on the Orlicz space $(L^{\chi},\| \cdot \|_{\chi})$ in section~\ref{sec:weaksol}.
It is important to note that in such an Orlicz space, the H\"older inequality (or H\"older--Young inequality) holds.

\begin{prop} [H\"older--Young inequality]\label{ine:Holder}
    For $f \in L^{\chi}$ and $g \in L^{\chi^{*}}$, we have
    \[
    \int_X |fg| dV_X \le \|f \|_{\chi} \| g\|_{\chi^{*}}.
    \]
\end{prop}
\begin{proof}
    See \cite[Proposition 1.3]{Dar_survay}.
\end{proof}
We will need the following additive Young inequality (also known as Fenchel--Young inequality).
\begin{prop}\label{Ineq: addYoung} 
    For every $a,b \in \R^+$, it holds that
    \[
    ab \le \chi(a)+\chi^*(b),\]
    with equality if and only if $b \in \partial \chi(a)$. 
\end{prop}
In particular, let $f,g : X \rightarrow \R^+$ be two measurable functions. Then
\[
f \left((\chi^*)^{-1} \circ g \right)\le \chi\circ f + g.
\]

\subsection{Envelopes}
Given a Lebesgue measurable function $ l: X \to \R$, the $\omega$-plurisubharmonic envelope of $l$ is defined by
\[
P_{\omega} (l ) \coloneqq \sup\{ u \in \PSH(X,\omega) : u \le l\}.
\]
The $(\omega,m)$-subharmonic envelope of $l$ is defined by
\[
P_{\omega,m} (l) \coloneqq \sup\{ u \in \SH_m(X,\omega) : u \le l\}.
\]
We now summarize the basic properties of these envelopes, following \cite{guedj2023degeneratecomplexhessianequations}.
\begin{prop} \label{prop: concentrated}If $l : X \to \R$ is continuous, then
    \begin{enumerate}
        \item The complex Monge--Amp\`ere measure $(\omega + \ddc P_{\omega}(l))^n$ is concentrated on the contact set $\{P_{\omega}(l)=l \}$.
        \item The complex Hessian measure $(\omega + \ddc P_{\omega,m}(l))^m \wedge \omega^{n-m}$ is concentrated on the contact set $\{P_{\omega,m}(l)=l \}$.
        \item If $l $ is $\mathcal{C}^{1,1}$-smooth, then $P_{\omega}(l )$ is $\mathcal{C}^{1,1}$-smooth and
        \[
        (\omega + \ddc P_{\omega}(l ))^n = \1_{\{P_{\omega} (l)=l\} } (\omega + \ddc l)^n.
        \]
        \item If $l$ is $\mathcal{C}^{1,1}$-smooth, then $P_{\omega,m}(l)$ is $\mathcal{C}^{1,1}$-smooth and
        \[
        (\omega + \ddc P_{\omega,m}(l))^m \wedge \omega^{n-m} = \1_{\{P_{\omega,m} (l)=l\} } (\omega + \ddc l)^m \wedge \omega^{n-m}.
        \]
    \end{enumerate}
\end{prop}
For detailed proofs and additional discussion, we refer the reader to \cite{Berman2019,ChuMcCleerey,ChuZhou2019,GL2022,GLZ2019envelopes,Tosatti_2018}.
As a consequence of the preceding results, we will make use of the following statements in the sequel.
\begin{lem} \label{lem:subsol preserve}
Fix $\lambda >0$. Let $u,v \in \SH_m(X,\omega) \cap L^{\infty}(X)$.
\begin{enumerate}
    \item If $ (\omega + \ddc u)^m \wedge \omega^{n-m} \ge e^{\lambda u}fdV_X$ and $ (\omega + \ddc v)^m \wedge \omega^{n-m} \ge e^{\lambda u }gdV_X$,
then \[(\omega+\ddc (\max(u,v)))^m \wedge \omega^{n-m} \ge e^{\lambda \max(u,v) }\min(f,g)dV_X.\]
\item If
$(\omega + \ddc u)^m \wedge \omega^{n-m} \le e^{\lambda u}fdV_X$, and $ (\omega + \ddc v)^m \wedge \omega^{n-m}\le e^{\lambda v}gdV_X$,
then \[(\omega+\ddc (P_{\omega,m}(\min(u,v)))^m \wedge \omega^{n-m} \le e^{\lambda P_{\omega,m}(\min(u,v) )}\max(f,g)dV_X.\]
\end{enumerate}
\end{lem}
The proof follows the same strategy as that for the Monge–Ampère measures presented in \cite[Lemma 1.9]{GL2022}.
\begin{proof}
Let $\psi \coloneqq \max(u,v)$. As a consequence of the maximum principle (Lemma~\ref{Lem:maxp}), we have 
\[
\begin{aligned}
(\omega + \ddc \psi)^m \wedge \omega^{n-m} \ge & \1_{\{ u \ge v\} } (\omega + \ddc u)^m \wedge \omega^{n-m}+\1_{\{ v >u\} }(\omega + \ddc v)^m \wedge \omega^{n-m} \\
\ge & \1_{\{ u \ge v\} }  e^{\lambda u}fdV_X +\1_{\{ v >u\} }e^{\lambda v}gdV_X \\
= & \1_{\{ u \ge v\} }  e^{\lambda \psi}fdV_X +\1_{\{ v >u\} }e^{\lambda \psi}gdV_X 
\ge   e^{\lambda \psi} \min(f,g) dV_X. 
\end{aligned}
\]
Set $\varphi \coloneqq P_{\omega,m}(\min(u,v))$. We prove the second statement. It follows from the proposition~\ref{prop: concentrated} that the complex Hessian measure $(\omega+\ddc \varphi)^m \wedge \omega^{n-m}$ has support in
\[
\left\{ P_{\omega,m}(\min(u,v)) =  \min(u,v)\right\} = \{ P_{\omega,m}(\min(u,v)) =  u <v\} \cup \{ P_{\omega,m}(\min(u,v)) =  v\}.
\]
We therefore obtain that
\begin{equation} \label{equa:maxp}
(\omega + \ddc \varphi)^m \wedge \omega^{n-m} \le \1_{ \{ \varphi =  u <v\}} (\omega + \ddc \varphi)^m \wedge \omega^{n-m} + \1_{\{ \varphi =  v\} } (\omega + \ddc \varphi)^m \wedge \omega^{n-m}.
\end{equation}
Moreover, since $\varphi \le u$ and $\varphi \le v$, the maximum principle (Lemma~\ref{Lem:maxp}) yields
\[
\1_{ \{ \varphi =  u\}} (\omega + \ddc \varphi)^m \wedge \omega^{n-m} \le \1_{ \{ \varphi =  u\}} (\omega + \ddc u)^m \wedge \omega^{n-m},
\]
and
\[\1_{ \{ \varphi =  v\}} (\omega + \ddc \varphi)^m \wedge \omega^{n-m} \le \1_{ \{ \varphi =  v\}} (\omega + \ddc v)^m \wedge \omega^{n-m}.\]
Together with \eqref{equa:maxp}, we deduce that
\[
\begin{aligned}
(\omega + \ddc \varphi)^m \wedge \omega^{n-m} \le &\1_{ \{ \varphi =  u <v\}} (\omega + \ddc u)^m \wedge \omega^{n-m} + \1_{\{ \varphi =  v\} } (\omega + \ddc v)^m \wedge \omega^{n-m}\\
\le& \1_{ \{ \varphi =  u <v\}} e^{\lambda u}fdV_X + \1_{\{ \varphi =  v\} } e^{\lambda v }g dV_X\\
\le& \1_{ \{ \varphi =  u <v\}} e^{\lambda \varphi}fdV_X + \1_{\{ \varphi =  v\} } e^{\lambda \varphi }g dV_X
\le e^{\lambda \varphi} \max(f,g) dV_X.
\end{aligned}
\]
\end{proof}

\section{Uniform \texorpdfstring{$L^{\infty}$-}{}estimates} \label{Sec:unifest}
In this section, we aim to establish a uniform $L^{\infty}$-estimate for solutions $\varphi \in \PSH(X,\omega) \cap L^{\infty}(X)$ to the Monge--Amp\`ere equation:
\begin{equation} \tag{MA}  \label{equa:globMA}
 (\omega + \ddc \varphi)^n \le f d V_X,  
\end{equation}
assuming $f \in L^{\chi}$, and the associated weight function $\chi$ satisfies \hyperref[Condition K]{Condition (K)}. We then generalize it to the complex Hessian equation.

For the sake of convenience, we fix 
\[
\chi(t) = t \log^n (t+1) (h \circ\log \circ \log (t+3))^n.
\]
We fix a constant $C_0 >0$ such that $\|f \|_{\chi} \le C_0 \int_X \chi(f)dV_X$ for all $f \in L^{\chi}$.
\subsection{Local \texorpdfstring{$L^{\infty}$-}{}bound of Monge--Amp\`ere operators }
Let $\Omega$ be the unit ball in $\C^n$. We define $g \coloneqq \frac{f }{ \| f\|_{\chi} }$ to normalize $f$. We now consider the following Dirichlet problem to find a plurisubharmonic solution $u$:
\begin{equation} \tag{LocMA} \label{equation:locMA}
   \begin{aligned}  
   (\ddc  u)^n &= g dV_X \text{, in }\Omega;\\
   u &= 0 \text{, on }\partial \Omega.
   \end{aligned}
\end{equation}
We now establish a local uniform estimate for $u$.
We start by establishing a useful result, which, for K\"ahler manifold, was shown in \cite[Theorem 3.3]{DDL21} and \cite[Lemma 2.2]{gl2025}. We present here a local version. Its proof is simpler than that of the global case, as plurisubharmonicity is preserved under addition.
\begin{lem} \label{lem: ineql est}
    Assume that $f \in L^p(\Omega, dV_X)$ for some $p >1$, and that $u$ is normalized so that $\sup_{\Omega}u =0$ and $u=0$ on $\partial \Omega$. Suppose that there exists a bounded function $v \in PSH(\Omega)$ normalized so that $\sup_{\Omega} v=0$ and $v=0$ on $\partial \Omega$, such that \[
    (\ddc u)^n \le (\ddc v)^n + fdV_X,
    \]
    then $ u \ge -C$, where $C=C(n,p, \|f\|_p, \|v\|_{\infty})$.
\end{lem}

\begin{proof}
     Let $\rho \in \PSH(\Omega)$ be the solution to the Monge--Amp\`ere equation \[
    (\ddc \rho)^n =f dV_X\text{ in }\Omega; \; \rho=0 \text{ on }\partial\Omega.
    \]
    The uniform bound for $\rho$ can be deduced from the results in \cite{Kolodziej1998,GL21}.
    Consider now $\psi =  v + \rho$. It is sufficient to show that $\psi \le u$.
    Indeed, $\psi$ satisfies
    \[
    (\ddc \psi)^n  \ge  (\ddc v)^n +  (\ddc \rho)^n =  (\ddc v)^n +  f dV_X \ge (\ddc u)^n.
    \]
    It thus follows from the Comparison Principle (Proposition~\ref{prop:CP}) that $u \ge \psi $ on $\Omega$.
\end{proof}

We are now prepared to prove a local uniform estimate.
\begin{lem}[Local uniform estimate for Monge--Amp\`ere equations] \label{lem:locMA}
    Assume that $u$ is a bounded plurisubharmonic solution to the Dirichlet problem \eqref{equation:locMA}. Then 
    \[
    \osc_{\Omega} u \le C,
    \]
    where $C=C(n, \Omega)$. 
\end{lem}

\begin{proof}
Let $v \in \mathcal{F}(\Omega)$ be such that
\[
(\ddc v)^n =  \chi \left( g \right) dV_X \text{, in }\Omega;\; v=0  \text{, on } \partial \Omega.
\]
The existence of $v\in \mathcal{F}(\Omega)$ is ensured by Theorem~\ref{thm:Cegrell98}. Let $\alpha>0$ be fixed (to be specified later). On the set $\left\{ \log \left( g+3 \right) < - \alpha v \right\}$ we have
\[
 (\ddc u)^n \le (e^{- \alpha v} -3 ) dV_X .
\]
We are going to show that $e^{-\alpha v }\in L^2$. To see this, observe that
\begin{equation} \label{equ:estlocv}
\int_{\Omega} (\ddc v)^n = \int_{\Omega} \chi\left( g\right) dV_X \le \chi(1).
\end{equation}
We now choose $\alpha = \frac{1}{\chi(1) }$. It follows from \cite[Theorem B]{ACK+} that
\[
\int_{\Omega} e^{-2 \alpha v } \le C_0(n,\Omega).
\]
We proceed to bound $\1_{\left\{\log  \left( g +3 \right) \ge - \alpha v \right\}} (\ddc u)^n$. 
Let $\Theta : \R^- \rightarrow \R^-$ be a convex increasing function. Then $\Theta \circ v$ satisfies 
\[
\Theta \circ v \in \PSH(\Omega) \text{, and } ( \ddc(\Theta \circ v ))^n=\left( \Theta''(v) dv \wedge d^c v + \Theta'(v) \ddc v\right)^n \ge (\Theta'\circ v)^n(\ddc v)^n.
\]
On the set $\left\{\log  \left( g +3\right) \ge -\alpha v  \right\}  $, we deduce that
\begin{equation} \label{equ:Thev}
( \ddc(\Theta \circ v ))^n \ge (\Theta'\circ v)^n(\ddc v)^n \ge 
\left(\Theta'\left( \frac{-\log  \left( g  +3\right)}{\alpha}\right)\right)^n  \chi \left( g \right) dV_X.
\end{equation}
We aim to construct a uniformly bounded function $\Theta$ such that, for every $t >1$, the following inequality is satisfied:
\begin{equation} \label{Formule: Conlam}
\left(\Theta'\left(\frac{-t}{\alpha}\right)\right)^n  t^n (h \circ \log t)^n \ge 2 .
\end{equation}
To see this, let $x = \log t$ and let 
\[
b(x)= - \Theta \left(\frac{-e^x }{\alpha} \right).
\]
Then the condition \eqref{Formule: Conlam} is equivalent to $b'(x)h(x) \ge \frac{2^{1/n}}{\alpha}$, for all $x>0$. We define $b$ by \[b(x) \coloneqq \int_0^{x}\frac{2^{1/n}}{\alpha h(s)} ds.\]
Since $h^{-1}$ is integrable, the function $b(x)$ remains bounded, which in turn implies that $\Theta$ is uniformly bounded, as claimed.
Taking $t=\log  \left( g +3\right)$ in \eqref{Formule: Conlam} yields that
\begin{equation} \label{formule:Theta}
\left(\Theta'\left(\frac{-\log  \left(g +3\right)}{\alpha}\right)\right)^n \left(\log  \left( g +3\right) \right)^n \left(h \circ \log \log  \left( g +3\right) \right)^n \ge 2.
\end{equation}
Without loss of generality, we assume that there exist $k_0 >0$ such that
\[
\chi(k)\ge \frac{1}{2}k\log^n (k+3) (h \circ\log \circ \log (k+3))^n, \;\forall k \ge k_0.
\]
From \eqref{equ:Thev} and \eqref{formule:Theta}, it follows that on the set $\{\log \left( g +3 \right) \ge -\alpha v \} \cap \{ g \ge k_0\}$, \[
( \ddc(\Theta \circ v ))^n \ge  g dV_X =  (\ddc u)^n.
\]
On the other hand, on the set $\{\log \left( g +3 \right) \ge -\alpha v \} \cap \{ g < k_0\}$, we deduce that
\[
(\ddc u)^n < k_0 dV_X.
\]
From the preceding arguments, we arrive at the result that
\begin{equation} 
(\ddc u)^n \le (e^{ -\alpha v} +k_0 ) dV_X  +  \left( \ddc(\Theta \circ v )\right)^n.
\end{equation}
It follows from Lemma~\ref{lem: ineql est} that $\osc_{\Omega} u \le C(n, \Omega )$.
\end{proof}

\subsection{Global \texorpdfstring{$L^{\infty}$-}{} bounds } We now provide a proof of the global $L^{\infty}$ bound for the solution to the Monge--Amp\`ere equations on Hermitian manifolds, that only relies on local solutions. The strategy is to construct a bounded subsolution by solving equations on each small ball, and then to show that the uniform boundedness of these subsolutions implies a uniform bound for the solution, using domination principle. This method is adapted from the work of Guedj--Lu, in \cite{GL21}.
We start by establishing a subsolution lemma.

\begin{lem}[Subsolution Lemma] \label{lem:subsol}
Assume $\chi$ satisfies \hyperref[Condition K]{Condition (K)}. Then 
there exist a uniform constant $s=s(\chi)>0$ such that for all $0\leq g\in L^{\chi}$, we can find $\psi\in \PSH(X,\omega)$, $-1\leq \psi \leq 0$, with
  \[
  (\omega +\ddc \psi)^n \geq  s(\chi) \frac{g}{\|g\|_{\chi}}  dV_X.
  \]  
\end{lem}

\begin{proof}
 Cover $X$ by balls, $X \subset \cup_{j=1}^N B_j$. Choose $B_j'=\{\rho_j <0 \}$, where $\rho_j: X \rightarrow \R$ are smooth functions such that $\omega+ \ddc \rho_j>0$ in a neighborhood of $B_j'$, and such that $B_j \Subset B_j^{'}$. In $B_j'$, let $v_j \in \PSH(B_j')$ be solution to Monge--Amp\`ere equation \[
(\ddc v_j)^n = \frac{g}{\|g\|_{\chi} } dV_X \text{ in }B_j'; \; v_j = -1 \text{ on }\partial \Omega.
\]
The local $L^{\infty}$-estimate (Lemma~\ref{lem:locMA}) implies $\osc_{B_j'}v_j \le C_j'$. We define $\psi_j \coloneqq \max (v_j, \xi_j \rho_j)$, where $\xi_j >1$ for all $j$. Thus, $\psi_j$ coincides with $v_j$ in $B_j$, and coincides with $\xi_j \rho_j$ in $X \setminus B_j'$ and in the neighborhood of $\partial B_j'$. Moreover, $\psi_j$ is $\delta^{-1}\omega$-plurisubharmonic for some uniform constant $\delta>0$ in $B_j$. Taking $\psi\coloneqq \frac{1}{N}\sum_{j=1}^N \delta \psi_j$, and taking $0 <s(\chi) \le \frac{\delta^n}{N^n}$, we infer that in $B_j$,
   \[
   \omega_{\psi} ^n \ge \frac{\delta^n}{N^n}(\omega+\ddc \psi_j)^n \ge s(\chi) (\ddc v_j)^n dV_X = s(\chi) \frac{g }{\|g\|_{\chi} }dV_X.
   \]
   Therefore, $\psi$ is the subsolution that we are looking for.   
\end{proof}

\begin{thm}[A priori estimates for Monge--Amp\`ere equations] \label{Thm:uniMA}
Let $\chi$ be as above, and let $f \in L^{\chi}$ with $f \ge 0$. Suppose that $\varphi\in \PSH(X,\omega)\cap L^{\infty}(X)$ satisfies 
  \[
  (\omega+dd^c \varphi)^n \leq fdV_X. 
  \]
  Let $s\coloneqq s(\chi)$ be the constant in Lemma \ref{lem:subsol} and assume that $A>0$ satisfies 
  \[
  \|\1_{\{f >A\} }f \|_{\chi}\leq \frac{s(\chi)}{4}.
  \]
  Then
  $Osc_X(\varphi)\leq C$, where $C$ depends on $A$, and $\chi$. 
\end{thm}
\begin{proof}
We consider the function \[
g = \frac{\1_{\{f >A\} }f }{\|\1_{\{f >A\} }f \|_{\chi}}.
\]
Applying the Subsolution Lemma (Lemma~\ref{lem:subsol}) to $g$ yields that there exists subsolution $\psi \in \PSH(X,\omega)$, such that
\[
(\omega + \ddc \psi)^n \ge s g dV_X.
\]
The dominated convergence theorem implies that
\[
\lim_{t \rightarrow + \infty}\int_X \chi(\1_{\{ f>t\} }f ) dV_X = \int_X \lim_{t \rightarrow + \infty} \chi \left(\1_{\{ f>t\} }f \right) dV_X  =0.
\]
We can thus choose $A$ sufficiently large, depending only on $f$ and $s$, such that \[ \| \1_{\{ f>A\}  } f\|_{\chi} \le \frac{s}{4} . \]
It follows that 
\[
(\omega + \ddc \psi )^n \ge 4 \1_{\{ f>A\}  } (\omega + \ddc \varphi)^n.
\]
Therefore,
\[
(\omega + \ddc \varphi)^n \le \frac{1}{4 }(\omega + \ddc \psi)^n + A dV_X.
\]
Fix $B>0$ large enough. It follows from \cite[Lemma 2.5]{GL21} that there exist $ \rho \in \PSH(X, \omega), -1 \le \rho \le 0$ such that 
\[
(\omega + \ddc \rho )^n \ge 4 A \1_{\{\varphi < -B\} } dV_X.  
\]
Normalize $\varphi$ such that $\sup_X \varphi =0$. Then on the set $\{ \varphi < \frac{\psi + \rho}{2}-B\} \subset \{ \varphi <-B\} $, we have
\[
(\omega + \ddc \varphi)^n \le \frac{1}{4 } (\omega + \ddc \psi)^n  + \frac{1}{4}(\omega+ \ddc \rho)^n \le \frac{1}{2} \left(\omega + \ddc \frac{\psi + \rho}{2}\right)^n.
\]
We conclude from the domination principle (Proposition~\ref{prop: DP}) that $\varphi \ge \frac{ \psi + \rho}{2}-B \ge -1-B$.
\end{proof}

\subsection{Global \texorpdfstring{$L^{\infty}$-}{} estimates for Hessian equations }

The uniform bound for the Hessian equation follows the method of \cite[Theorem 3.1]{guedj2023degeneratecomplexhessianequations}, which assumes $f \in L^p$ with $p > n/m$. A generalization to Orlicz spaces is outlined in Remark 3.2 therein. It has also been shown in the recent article \cite[Theorem B]{gl2025} that this approach can be adapted in certain general complex geometric PDEs. We provide a more detailed presentation in the proof of \hyperlink{ThmB}{Theorem B} here.

The following lemma will be used to obtain a uniform estimate for bounded solutions.
\begin{lem} \label{Lem:density}
    Assume $f \in L^1(dV_X) $. Let $\varphi  \in \SH_m(X,\omega) \cap L^{\infty}(X)$ be such that 
    \[
    (\omega + \ddc \varphi)^{m}\wedge \omega^{n-m} = f \omega^n.
    \]
    Then, we have $(\omega+ \ddc u)^n \le  f^{n/m} \omega^n$, where $u = P_{\omega}(\varphi) $.
\end{lem}
\begin{proof} 
    The proof proceeds in two steps. We first treat the case where the density $f$ is bounded. In the general case, we approximate a given $f\in L^{1}$ by a sequence of bounded densities. Normalize the solution $\varphi$ such that $\sup_X \varphi=0$.
    
    \emph{Step 1.} Assume that $f$ is bounded. Let $\{f_j\}$ be a sequence of smooth, uniformly bounded functions such that $f_j \to f$ in $L^{2n}$. It follows from \cite[Lemma 3.20]{Ko_odziej_2016} that there exists a decreasing sequence of $(\omega,m)$-subharmonic functions $\{ \psi_j\}$ such that $\psi_j \to \varphi$. We consider the following complex Hessian equation:
    \[
    (\omega + \ddc \varphi_j)^m \wedge \omega^{n-m} = e^{\varphi_j - \psi_j} f_j \omega^n.
    \]
    The existence of smooth, $(\omega,m)$-subharmonic solution $\varphi_j$ follows from \cite{Sze18}. Then, from \cite[Theorem 3.9]{Ko_odziej_2016} we deduce that $\varphi_j$ converges uniformly to $\widetilde{\varphi}\in \SH_m(X,\omega) \cap \mathcal{C}^0(X)$, which is the solution to
    \[
    (\omega + \ddc \widetilde{\varphi})^m \wedge \omega^{n-m} = e^{\widetilde{\varphi} - \varphi} f \omega^n.
    \]
    By the domination principle (Proposition~\ref{prop: DP}), we obtain that $\widetilde{\varphi} = \varphi$.
    Taking $u_j = P_{\omega}(\varphi_j)$, it follows from Proposition~\ref{prop: concentrated} that $u_j \in \mathcal{C}^{1,1}(X)$. The mixed Monge--Amp\`ere inequality \cite[Lemma 1.19]{NG16} implies that
    \[
    (\omega + \ddc u_j)^n \le e^{\frac{n(\varphi_j - \psi_j)}{m}} f_j^{n/m} \omega^n.
    \]
    Since $u_j$ converges uniformly to $u$, passing to the limit, we conclude that \[(\omega + \ddc u)^n \le  f^{n/m} \omega^n.\]

    \emph{Step 2.} Assume now that $f \in L^{1}(dV_X)$. We consider the equation given by
    \[
    (\omega + \ddc \varphi_j)^m \wedge \omega^{n-m} = e^{\varphi_j - \varphi} \min(f,j) \omega^n.
    \]
   Theorem 3.19 of \cite{KN15weaksol} ensures the existence of a bounded solution $\varphi_j$, satisfying $\varphi_j \ge \varphi$. The domination principle yields that $\varphi_j$ is decreasing to a $(\omega,m)$-subharmonic function $\widetilde{\varphi}$. It follows from the domination principle again that $\widetilde{\varphi}=\varphi$. Taking $u_j \coloneqq P_{\omega}(\varphi_j)$, then $u_j$ decreases to $u$. Moreover, Step 1 yields \[
    (\omega + \ddc u_j)^n \le e^{n/m(\varphi_j - \varphi)} f_j^{n/m} \omega^n.
    \]
    Applying the continuity of Monge--Amp\`ere operator along monotone sequences \cite{BedfordTaylor1976,BedfordTaylor1982}, we obtain
    \[
    (\omega + \ddc u)^n \le  f^{n/m} \omega^n,
    \]
    as required.
\end{proof}

\begin{proof}[Proof of \hyperlink{ThmB}{Theorem B}]
    Normalize $\varphi$ such that $\sup_X \varphi=0$. Let $u=P_{\omega} \left( \varphi\right)$ and let $\mathcal{C}$ denote the contact set $\mathcal{C} \coloneqq \{u=\varphi\}$. We remark that $u \in \PSH(X, \omega) \subset \SH_m(X, \omega)$. Therefore, it follows from the mixed Monge--Amp\`ere inequality \cite[Lemma 1.19]{NG16} and proposition~\ref{prop: concentrated} that 
    \[
    \1_{\mathcal{C}} (\omega +\ddc u)^m \wedge \omega^{n-m} \le \1_{\mathcal{C}} (\omega + \ddc \varphi)^m \wedge \omega^{n-m} \le f dV_X
    \]
We remark that $h_X \coloneqq \omega^n / d V_X \ge \delta_0$ is positive on X. Lemma~\ref{Lem:density} ensures us to take $g \coloneqq \left(\omega+d d^c u\right)^n / d V_X$. Then, the mixed Monge--Amp\`ere inequality yields that
\[
\delta_0^{1-m/n} g^{m/n} \le h_X^{1-m/n} g^{m/n} \le f,
\]
hence $g \le C_0 f^{n/m}$, so that $g \in L^{\chi_0}$, where $\chi_0=  t \log^n(1+t) [h \circ \log \circ \log (3+t)]^n $ satisfies \hyperref[Condition K]{Condition (K)}. It follows from Theorem~\ref{Thm:uniMA} that
\[
\osc_X(u) \le A.
\]
It suffices to bound $\sup_X u$ from below. Set \[v^{-}_A(\omega) \coloneqq \inf \left\{ \int_X (\omega + \ddc \varphi)^n: \varphi\in L^{\infty}(X) \cap \PSH(X, \omega), -A \le \varphi \le 0\right\}.\] Recall that $v_A^{-}(\omega)$ is strictly positive by \cite[Proposition 3.4]{GL2022}. We denote by $(\chi_0^{*})^{-1}$ the inverse of $\chi_0^{*}$. Then it follows from additive H\"older--Young inequality (Proposition~\ref{Ineq: addYoung}) that
$$
\begin{aligned}
v^{-}_A(\omega)(\chi_0^{*})^{-1}\left(-\sup_X u\right) & \le \int_X (\chi_0^{*})^{-1}(-u)\left(\omega+d d^c u\right)^n \\
&= \int_X (\chi_0^{*})^{-1}(-u)g d V_X  \\
&\le \int_X \chi_0(g) d V_X + \int_X (-u) d V_X.
\end{aligned}
$$
We remark that $\int_{X} (-u)dV_X$ is uniformly bounded from above by Lemma~\ref{lem:comp}. Thus we obtain a uniform lower bound of $\sup_X u$, which yields together with $\osc_X u \le A$ that $|u|$ has a uniform bound, and since $u \le \varphi$, the solution $\varphi$ is also uniformly bounded.

\end{proof}

\section{Weak solution} \label{sec:weaksol}

\subsection{Stability estimates}
\begin{thm} \label{Thm:stab}
    Let $u,v \in \SH_m(X,\omega)\cap L^{\infty}(X)$ be such that $\sup_X u =0$ and $v \le 0$. Assume that $u \in \SH_m(X,\omega)$ is a solution to the complex Hessian equation 
    \begin{equation} \label{eqn: stab}
     \omega_u^m \wedge \omega^{n-m}= e^{\lambda u} f \omega^n,   
    \end{equation}
     where $\lambda\ge0$ and $f^{n/m} \in L^{\chi}$ with $\chi$ satisfying \hyperref[Condition K]{Condition (K)}.
    Then,
    \[
    \sup_X (v-u)_{+} \le C \| (v-u)_+f ^{n/m}\|_{\chi}^{\frac{1}{n+1}},
    \]
    where the positive constant $C$ depends only on $n,m$, $\omega$, $\|f^{\frac{n}{m} }\|_{\chi}$, and $\sup_X|v|$.
\end{thm}

\begin{proof}
We consider only the case  $\lambda =0$; when $\lambda >0$, the $L^{\infty}$-estimate implies that $e^{\lambda u}f \le f $, allowing the problem to be reduced to the case $\lambda=0$. 
We now fix $\epsilon>0$, to be chosen later, and define $g_{\epsilon} \coloneqq \1_{\{u <v -\epsilon \} } f^{n/m}$.

By the subsolution lemma (Lemma~\ref{lem:subsol}), there exist an uniform constant $s \coloneqq s(\chi)$, and $-1 \le \psi \le 0$ such that
\[
(\omega+dd^c \psi)^n \geq s \frac{g_{\epsilon}dV_X}{\|g_{\epsilon}\|_{\chi} } .
\]
It follows from mixed Monge--Amp\`ere inequality \cite[Lemma 1.9]{NG16} that 
\[
(\omega+dd^c \psi)^{m} \wedge \omega^{n-m} \geq s^{m/n}  \frac{\1_{\{u <v - \epsilon \} } f}{\|g_{\epsilon}\|_{\chi}^{m/n}}.
\]
We set $\varphi \coloneqq (1- \delta)v + \delta \psi$, where $0<\delta<1$ is to be specified later. Then there exists a positive constant $C_2$, depending only on $\| \psi \|_{\infty}$ and $\sup_X|v|$ such that
\[
\{\varphi > u + C_2\delta +\epsilon\} \subset \{u < v -  \epsilon\}.
\]
From the $(\omega,m)$-subharmonicity of $v$ and $\psi$, and applying \cite[Lemma 2.3]{KN15weaksol}, we derive that
\[
\begin{aligned}
    \1_{\{\varphi > u + C_2 \delta +  \epsilon\}} \omega_{\varphi}^m \wedge \omega^{n-m} =& \1_{\{\varphi > u + C_2 \delta +  \epsilon\}} \left[ (1-\delta)\omega_v +\delta \omega_{\psi}\right]^m \wedge \omega^{n-m} \\
    =& \1_{\{\varphi > u + C_2 \delta +  \epsilon\}} \sum_{k=0}^m \binom{m}{k} \left[(1-\delta)\omega_v\right]^k \wedge (\delta \omega_{\psi})^{m-k} \wedge \omega^{n-m}\\
    \ge & \1_{\{ \varphi > u + C_2 \delta +  \epsilon\}} \delta^{m} (\omega_{\psi})^m \wedge \omega^{n-m} .
\end{aligned}
\]
If $ \|g_{\epsilon}\|_{\chi}$ is not small, then
the stability property follows trivially by adjusting $C$. Hence, we can assume that $\|g_{\epsilon}\|_{\chi} < \frac{s}{2^{n/m}}$. Choosing $\delta = \frac{2^{1/m} \|g_{\epsilon}\|_{\chi}^{1/n} }{s^{1/n}} $, we obtain that 
\[
\begin{aligned}
\1_{\{ \varphi > u + C_2 \delta + \epsilon\}} \delta^{m} (\omega_{\psi})^m \wedge \omega^{n-m}  
\ge& 2 \1 _{\{ \varphi > u + C_2 \delta + \epsilon\}}f \omega^n \\
=& 2 \1 _{\{ \varphi > u + C_2 \delta + \epsilon\}} \omega_u^m \wedge \omega^{n-m}.
\end{aligned}
\]
It follows from domination principle (Proposition~\ref{prop: DP}) that
\[
\varphi \le u + C_2\delta +\epsilon.
\]
Hence,
\[ \label{formule:stabes}
\begin{aligned}
(v-u )_{+}\le & C_3 \left(\frac{2^{1/m}}{s^{1/n}}\right)\|g_{\epsilon}\|_{\chi}^{1/n} +  \epsilon\\
\le & C_4 \|g_{\epsilon}\|_{\chi}^{1/n} +  \epsilon\\
\le & C_4 ( \epsilon)^{-1/n} \|(v-u)_{+} f^{n/m} \|^{1/n}_{\chi} +\epsilon.
\end{aligned}
\]
Optimizing the right hand side by choosing 
\[
\epsilon = \left(\frac{ \|C_4(v-u)_{+} f^{n/m} \|^{1/n}_{\chi}}{n} \right)^{\frac{n}{n+1}},
\]
we obtain
\[
(v-u)_{+} \le C \left( \|(v-u)_{+} f^{n/m} \|_{\chi}\right)^{\frac{1}{n+1}}.
\]
The conclusion follows by taking the supremum of the left-hand side.
\end{proof}

\begin{remark} \label{rem: Lpstability}
    When $\chi(t)= t^p/p$, with $p >1$, the Orlicz norm $ \|\cdot\|_{\chi}$ is simply the Lebesgue $L^p$-norm. Assume $f^{n/m} \in L^{p}$, for $p>1$. Let $\chi_q(t)= t^q/q$ for some $1<q<p$. Then $\psi \in \PSH(X,\omega)$ is a subsolution such that 
    \[
    (\omega + \ddc \psi)^{n}  \ge s'g \omega^n, \text{ where } g =\frac{\1_{\{u < v- B\epsilon \}}  f^{n/m}}{\| \1_{\{u < v- B\epsilon \}}  f^{n/m}\|_{q} } .
    \]
    Following the same procedure and applying the Chebyshev inequality to \[(v-u )_{+} \le C_4 \|(v-u )_{+} g_{\epsilon}\|_q^{1/n} +\epsilon,
    \]we obtain that
    \[
    (v-u )_{+} \le \frac{C'}{( \epsilon)^{\frac{p-q}{p q n}}} \|(v-u )_{+}\|^{\frac{p-q}{pqn}}_{1} \|f^{n/m}\|^{1/n}_p +  \epsilon.
    \]
    Choosing an optimal $\epsilon = C'_1 \|(v-u )_{+}\|^{\frac{p-q}{  (npq+p-q)}}_{1} \|f^{n/m}\|^{\frac{pq}{(npq+p-q)}}_p$, we deduce a stability estimate \[
    (v-u )_{+} \le  C^{''} \|(v-u )_{+}\|_1^{\frac{p-q}{(npq+p-q)}} \left\|f^{\frac{n}{m} }\right\|^{\frac{pq}{(npq+p-q)}}_p,
    \]
    which slightly generalize the stability estimate in \cite{Ko_odziej_2016}.
\end{remark}
The following result is derived by applying the above theorem twice.
\begin{coro} \label{cor:stab}
  Suppose that $u,v \in \SH_m(X,\omega)$, normalized such that $\sup_X u=0$, $\sup_X v =0$, satisfy the complex Hessian equations:
  \[
  \omega_u^m \wedge \omega^{n-m} = e^{\lambda u}f \omega^n, \omega_v^m \wedge \omega^{n-m} = e^{\lambda v}g \omega^n,
  \]
  where $0 < f^{\frac{n}{m} },g^{\frac{n}{m} } \in L^{\chi}$, and $\lambda\ge 0$. Then,
  \[
  \| u-v\|_{\infty} \le C\left( \| (v-u)_+f^{\frac{n}{m} }\|_{\chi}^{\frac{1}{n+1}} + \| (u-v)_+g^{\frac{n}{m} }\|_{\chi}^{\frac{1}{n+1}} \right),
  \]
  where the constant $C$ depends on $ n, m, \omega$, $\|f^{\frac{n}{m} }\|_{\chi}$ and $\|g^{\frac{n}{m} }\|_{\chi}$.
\end{coro}

\subsection{Continuity of solutions}
This subsection is devoted to the proof of \hyperlink{ThmC}{Theorem C}. We approximate $f$ by positive densities \[f_j= \min (f,j).\] It follows from the main theorem of \cite{Ko_odziej_2016} that for each $j$, there exist a continuous solution $u_j$ and a constant $c_j>0$ such that
\[
(\omega + \ddc u_j)^m \wedge \omega^{n-m} = c_j e^{\lambda u_j}f_j dV_X, \;\lambda\ge 0.
\]
We normalized $u_j$ such that $\sup_Xu_j=0$. Our first step is to bound the coefficients $c_j$ from above and below. This requires an inequality involving mixed Hessian measures.
It's well known that if ${\varphi}_j\in \SH_m(X,\omega)\cap \mathcal{C}^{\infty}(X)$ and $(\omega+\ddc {\varphi}_j)^m \wedge \omega^{n-m} = f_j dV_X$, then G\aa rding's inequality \cite{gar59} holds:
\[
(\omega +\ddc {\varphi}_1)\wedge\cdots \wedge (\omega+ \ddc {\varphi}_m)\wedge \omega^{n-m} \ge (f_1\cdots f_m)^{1/m}dV_X.
\]
In the case where $\varphi_j$'s are not smooth, we can apply an approximation argument to obtain the following inequality. 

\begin{prop} \label{prop:mixedHess}
    Let $u \in \SH_m(X,\omega)$ be bounded functions such that 
    \[
    (\omega + \ddc {\varphi})^m \wedge \omega^{n-m} \ge f\omega^n.
    \]
    Then we have that $(\omega + \ddc {\varphi}) \wedge \omega^{n-1} \ge f^{1/m} \omega^n$.
\end{prop}
\begin{proof}
    Let $\{g_j\}$ be a sequence of bounded, positive and smooth functions converging to $f$. Let $v_j$ be a sequence of smooth, $(\omega,m)$-subharmonic functions converging to ${\varphi}$. We consider the following Hessian equation:
    \[
    (\omega + \ddc {\varphi}_j)^m \wedge \omega^{n-m} = e^{{\varphi}_j - v_j}g_j \omega^n.
    \]
    The existence of smooth solution ${\varphi}_j$ follows from \cite{Sze18}. Applying G\aa rding's inequality \cite{gar59} yields
    \[(\omega + \ddc {\varphi}_j) \wedge \omega^{n-1} \ge e^{\frac{{\varphi}_j - v_j}{m}}g_j^{1/m} \omega^n.\]
    Moreover, we deduce from \cite[Theorem 3.19]{Ko_odziej_2016} that ${\varphi}_j$ converges uniformly to $\widetilde{{\varphi}} \in \SH_m(X,\omega)$. Passing to the limit as $j \to +\infty$, we obtain
    \[
    (\omega + \ddc \widetilde{u} )^m \wedge \omega^{n-m} = e^{\widetilde{{\varphi}} -{\varphi}}f \omega^n.
    \]
   The domination principle (Proposition~\ref{prop: DP}) then implies $\widetilde{{\varphi}}={\varphi}$. Consequently, we conclude that
    \[
    (\omega +\ddc {\varphi}) \wedge \omega^{n-1} \ge f^{1/m} \omega^n.
    \]
\end{proof}
We now turn to obtaining uniform bounds for $c_j$.
\begin{lem} \label{lem:bdd cj}
 The sequence of constant $\{ c_j\}$ is uniformly bounded. In particular, the Orlicz norm $\| c_j f_j\|_{\chi}$ is uniformly bounded from above.
\end{lem}

\begin{proof}
    We first bound $c_j$ from above. Fix $\delta_0 >0$ such that $\delta_0 \le \omega^n /dV_X$. Proposition~\ref{prop:mixedHess} implies that for each $j$,
    \[
    (\omega + \ddc u_j) \wedge \omega^{n-1} \ge c_j^{1/m} f_j^{1/m}  \omega^n \ge \delta_0 c_j^{1/m} f_j^{1/m} dV_X.
    \]
    Thus, by taking the integral of both sides:
    \[
     \delta_0 c_j^{1/m} \int_X f_j^{1/m} dV_X \le \int_X(\omega+ \ddc u_j) \wedge \omega^{n-1}.
    \]
    Note that $u_j \in \SH_m(X, \omega) \subset \SH_1(X, \omega)$, and the sequence $(u_j)$ is relatively compact in $L^1$. Moreover, we have $\int_X f_j^{1/m} \omega^n \rightarrow \int_X f^{1/m} \omega^n >0$ as $j \rightarrow +\infty$. It follows that there exist $j_0$ large enough such that for each $j > j_0$, the following inequality holds:
    \[
    c_j^{1/m} \le \frac{2}{\delta_0 \int_X f^{1/m} dV_X} \int_X (\omega+ \ddc u_j)\wedge \omega^{n-1}.
    \]
    It suffices to bound the right hand side. Indeed, it follows from Stokes' theorem that
    \[
    \int_X (\omega+ \ddc u_j)\wedge \omega^{n-1} = \int_X \omega^n + \int_X u_j \ddc(\omega^{n-1})  \le \int_X \omega^n + B\int_X |u_j |\omega^{n}.
    \] 
    By the $L^1$-compactness result for $(\omega,1)$-subharmonic functions (Lemma~\ref{lem:comp}), we conclude that the right hand side above is bounded. Thus $c_j$ is bounded from above.

    We next show that $\{c_j\}$ is bounded away from $0$. Let $g_j= f_j^{n/m}\in L^{\chi}$. By the subsolution Lemma (Lemma~\ref{lem:subsol}), there exists a uniform constant $\delta>0$, such that 
    \[(\omega+\ddc v_j)^n  \ge \delta_j g_j  dV_X.
    \]
    Then, it follows from mixed Monge--Amp\`ere inequality \cite[Lemma 1.9]{NG16} that
    $$
     \omega_{v_j}^m \wedge \omega^{n-m} \ge \delta_j^{m/n} g_j^{m/n} \omega^n  \ge  \delta_j^{m/n} f_j \omega^n.
    $$
    The domination principle (Proposition~\ref{prop: DP}) yields that 
    \[
    \delta_j^{m/n} \le c_j.
    \]
    We thus obtain a uniform lower bound for $c_j$'s, since $\delta_j$'s are strictly positive, uniform constants. The uniqueness of $c_j$ for each $f_j$ follows directly from domination principle (Proposition~\ref{prop: DP}).
    \end{proof}
\emph{Uniform bound of $u_j$.}
It follows from \hyperlink{ThmB}{Theorem B} that $\{ u_j\}$ have uniform bound:
\[
-C \le u_j \le 0. 
\]
By extracting, we can assume that $\{ u_j \}$ converges in $L^1(X)$ to a bounded $(\omega,m)$-subharmonic function $u$, as $j \rightarrow +\infty$. 

\emph{Existence of solutions.} We now show that, by construction, the limit $u$ is a solution to the Hessian equation
\[
(\omega+\ddc u)^m \wedge \omega^{n-m} = ce^{\lambda u}f dV_X, \lambda \ge 0.
\]
We first consider the simpler case in which $\lambda>0$.
    \begin{thm} \label{Thm: exis bdd sol 0}
Let $\lambda>0$. Then $u \in \SH_m(X,\omega) \cap L^{\infty}(X)$ is the unique solution such that
\[
(\omega + \ddc u )^m \wedge \omega^{n-m} = e^{\lambda u} f dV_X.
\]
Moreover $u$ is uniformly bounded.
\end{thm}
\begin{proof}
    To simplify the notation, we assume that $\lambda=1$. Then each $u_j$ satisfies
    \[
    (\omega+ \ddc u_j)^m \wedge \omega^{n-m} = e^{u_j} f_j dV_X.
    \]
    From the domination principle (Proposition~\ref{prop: DP}), we deduce that $u_j$ decreases to some functions $u \in \SH_m(X,\omega)$. Passing to the limit, it follows from the continuity of Hessian operator along decreasing sequence (\cite[Section 5]{kolodziej2023complexhessianmeasuresrespect}) that 
    \[
    (\omega+ \ddc u)^m \wedge \omega^{n-m} = e^{ u} fdV_X.
    \]
    Applying \hyperlink{ThmB}{Theorem B}, we obtain a uniform bound for $u$.
\end{proof}
Monotonicity of the sequence $u_j$ is not guaranteed when $\lambda=0$. Our approach is to construct a supersolution and a subsolution and demonstrate that they coincide, as in \cite[Section 3]{GL21a} and \cite[Theorem 4.5]{BGL24}.
\begin{thm}[Existence of bounded solution] \label{Thm: exis bdd sol}
Let $f$, $u$ and the constant $c$ be as above. Then $u \in \SH_m(X,\omega) \cap L^{\infty}(X)$ is a solution to
\[
(\omega + \ddc u)^m \wedge \omega^{n-m} = c f dV_X.
\]
Moreover, $u$ is uniformly bounded.
\end{thm}
\begin{proof}
    We define
    \[
    \rho_j \coloneqq \left(\sup_{k \ge j} u_k \right)^*, \; v_j \coloneqq P_{\omega,m}\left(\inf_{k \ge j}u_k\right),
    \]
    where $P_{\omega,m}\left(\inf_{k \ge j}u_k\right)$ denotes the largest $(\omega,m)$-subharmonic function lying below $\inf_{k \ge j}u_k$. Observe that $\rho_j $ is a decreasing sequence that converges to $u$, and $v_j$ is an increasing sequence that converges to an $(\omega,m)$-subharmonic function $v$. In particular, by construction, we have $u \ge v$. We are going to show that $u=v$.

   Set $c_j^- f_j^- \coloneqq \inf_{k\ge j} c_k f_k$. Then it follows from  Lemma~\ref{lem:subsol preserve} that 
    \[
    (\omega + \ddc \rho_j)^m \wedge \omega^{n-m} \ge c_j^- f_j^- dV_X.
    \]
    Then we deduce from the continuity of Hessian operator along monotone sequences (see \cite[Lemma 5.1]{kolodziej2023complexhessianmeasuresrespect}) that 
    \begin{equation} \label{eqn:supersol}
    (\omega + \ddc u)^m \wedge \omega^{n-m} \ge c f dV_X.
     \end{equation}
    On the other hand, we write $(\omega+ \ddc v_j)^{m}\wedge \omega^{n-m} = e^{v_j - v_j}c_j f_j dV_X$. Define $c_j^+ f_j^+ \coloneqq \sup_{k\ge j} c_k f_k$. It also follows from Lemma~\ref{lem:subsol preserve} that 
    \[
    (\omega + \ddc v_j)^m \wedge\omega^{n-m} \le e^{v_j - \inf_{k \ge j}u_k} c_j^{+} f_j^+ dV_X.
    \]
    Let $j \to +\infty$. Together with \eqref{eqn:supersol} we deduce that 
    \[
    (\omega+ \ddc v)^m \wedge \omega^{n-m} \le e^{v- u}cf dV_X \le e^{v-u} (\omega+ \ddc u)^m \wedge \omega^{n-m} .
    \]
    By the domination principle (Proposition~\ref{prop: DP}), we conclude that $v \ge u$. Hence, $v = u$ is a bounded solution to the complex Hessian equation:
    \[
    (\omega + \ddc u)^m \wedge \omega^{n-m} = c f dV_X.
    \]
    The uniform bound for $u$ follows from \hyperlink{ThmB}{Theorem B}.
\end{proof}

\begin{proof}[Proof of \hyperlink{ThmC}{Theorem C}]

\emph{Existence of continuous solutions.}  We now prove that the sequence $\{ u_j \}$ converges uniformly to $u $ as $j \rightarrow +\infty$. We recall that when $\lambda >0$, the $L^{\infty}$-estimate implies that $e^{\lambda u }f \le f$, allowing us to reduce to the case where $\lambda =0$. Fix $j_0$ large enough. For each $j>j_0$, we infer by Theorem~\ref{Thm:stab} that\[
(u_j - u)_+ \le C_1  \left\|(u_j - u)_+f^{\frac{n}{m} }  \right\|_{\chi} ^{\frac{1}{n+1}}, \; (u - u_j )_+ \le C_1 \left\|(u - u_j)_+f_j^{\frac{n}{m} }  \right\|_{\chi} ^{\frac{1}{n+1}}.
\]
We are going to show that $\left\|(u_j - u)_+f ^{\frac{n}{m} }  \right\|_{\chi} \rightarrow 0$ as $j \rightarrow +\infty$. Fixing $s >0$, we proceed to compute the corresponding integrability.
\small
\[
\begin{aligned}
&\int_X \chi \left( \frac{(u_j-u)_+f^{\frac{n}{m} }}{s} \right) dV_X 
\\= & \int_X  \left( \frac{(u_j-u)_+f^{\frac{n}{m} }}{s} \right) \log^n \left( 1+\frac{(u_j-u)_+f^{\frac{n}{m} } }{s} \right) \left(h \circ \log \circ \log \left( 3+\frac{(u_j-u)_+f^{\frac{n}{m} }}{s} \right)\right)^ndV_X.
\end{aligned}
\]
As $u$ and $u_j$'s are uniformly bounded, it follows that $(u_j -u)+ \le C_0$ for some uniform constant $C_0$, and since $h$ is increasing, we deduce that
\begin{multline*}
\int_X \chi \left( \frac{(u_j-u)_+f^{\frac{n}{m} }}{s} \right) dV_X \le  \int_X  \left( \frac{(u_j - u)_+f^{\frac{n}{m} }}{s} \right) \log^n \left( 1+\frac{C_0 f^{\frac{n}{m} }}{s} \right) \\ \cdot \left(h \circ \log \circ \log \left( 3+\frac{C_0f^{\frac{n}{m} }}{s} \right)\right)^n dV_X.
\end{multline*}

By the dominated convergence theorem, we conclude that \[
\left\|\frac{(u_j - u)_+ f^{\frac{n}{m} }}{s} \right\|_{\chi} \rightarrow 0 \text{, for all } 0<s<+\infty,
\] 
and consequently \[
\|(u_j - u)_+f^{\frac{n}{m} } \|_{\chi} ^{\frac{1}{n+1}} \rightarrow 0.\]
Applying the same argument, we obtain that $\left\|(u - u_j)_+f_j^{\frac{n}{m} }  \right\|_{\chi} ^{\frac{1}{n+1}} \le \left\|(u - u_j)_+f^{\frac{n}{m} }  \right\|_{\chi} ^{\frac{1}{n+1}} \rightarrow 0$.
Therefore, $ \{u_j\}_{j=1}^{\infty}$ is a Cauchy sequence in $\SH_m(X,\omega)\cap \mathcal{C}^0(X)$. It follows that the bounded solution $u$ to the complex Hessian equation
\[
(\omega + \ddc u )^m \wedge \omega^{n-m} = c e^{\lambda u} f \omega^n, \text{ where }f \in L^{\chi} ,
\]
is continuous.

\emph{Continuity of bounded solutions.} Assume now that $v \in \SH_m(X,\omega)$ is another bounded solution to the complex Hessian equation
\[
(\omega + \ddc v)^m \wedge \omega^{n-m} = c f dV_X.
\]
Then $v$ is continuous. Indeed, the uniform bound for $v$ ensures that $\widetilde{f} = e^{-v}f \in L^{\chi}$. We deduce from the previous argument that there exist $\widetilde{u} \in \SH_m(X,\omega)\cap \mathcal{C}^0(X)$ solving
\[
(\omega + \ddc \widetilde{u})^m \wedge \omega^{n-m} = e^{\widetilde{u}}\widetilde{f}dV_X.
\]
It follows that 
\[
(\omega + \ddc \widetilde{u})^m \wedge \omega^{n-m} = e^{\widetilde{u}-v} f dV_X =  e^{\widetilde{u}-v} (\omega + \ddc v)^m \wedge \omega^{n-m}.
\]
Applying the Domination Principle (Proposition~\ref{prop: DP}) yields that $v = \widetilde{u}$. Therefore, $v \in \mathcal{C}^0(X)$, as claimed.
\end{proof}

\bibliographystyle{alpha}
\bibliography{ref}

\newcommand{\etalchar}[1]{$^{#1}$}
\begin{thebibliography}{GPTW21}

\bibitem[{\AA}C25]{AC24}
Per {\AA}hag and Rafał Czyż.
\newblock Continuity of solutions to complex {Hessian} equations via the dinew–{Kołodziej} estimate.
\newblock {\em Ann. Fenn. Math.}, 50:201--214, 2025.

\bibitem[{\AA}CK{\etalchar{+}}09]{ACK+}
Per {\AA}hag, Urban Cegrell, S{\l}awomir Ko{\l}odziej, Hoang-Hiep Ph{\d{a}}m, and Ahmed Zeriahi.
\newblock Partial pluricomplex energy and integrability exponents of plurisubharmonic functions.
\newblock {\em Adv. Math.}, 222:2036--2058, 01 2009.

\bibitem[Ber19]{Berman2019}
Robert~J. Berman.
\newblock From {Monge--Amp{\`e}re} equations to envelopes and geodesic rays in the zero temperature limit.
\newblock {\em Math. Z.}, 291:365--394, 2019.

\bibitem[BGL25]{BGL24}
S{\'e}bastien Boucksom, Vincent Guedj, and Chinh~H. Lu.
\newblock Volumes of bott--chern classes.
\newblock {\em Peking Math. J.}, 06 2025.

\bibitem[B{\l}o05]{Blocki05}
Zbigniew B{\l}ocki.
\newblock Weak solutions to the complex {H}essian equation.
\newblock {\em Ann. Inst. Fourier (Grenoble)}, 55(5):1735--1756, 2005.

\bibitem[B{\l}o12]{Błocki2012Lecture}
Zbigniew B{\l}ocki.
\newblock {\em The Calabi--Yau Theorem}, pages 201--227.
\newblock Springer Berlin Heidelberg, Berlin, Heidelberg, 2012.

\bibitem[BT76]{BedfordTaylor1976}
Eric Bedford and B.~A. Taylor.
\newblock The dirichlet problem for a complex {Monge-Ampère} equation.
\newblock {\em Invent. Math.}, 37:1--44, 1976.

\bibitem[BT82]{BedfordTaylor1982}
Eric Bedford and B.~A. Taylor.
\newblock A new capacity for plurisubharmonic functions.
\newblock {\em Acta Math.}, 149:1--40, 1982.

\bibitem[CC21]{chencheng2021cscK1}
Xiuxiong Chen and Jingrui Cheng.
\newblock On the constant scalar curvature kähler metrics (i)—a priori estimates.
\newblock {\em J. Amer. Math. Soc.}, 34(4):909--936, 2021.

\bibitem[Ceg98]{Cegrell1998}
Urban Cegrell.
\newblock Pluricomplex energy.
\newblock {\em Acta Math.}, 180(2):187--217, 1998.

\bibitem[Ceg04]{Cegrell04}
Urban Cegrell.
\newblock The general definition of the complex {Monge-Amp\`ere} operator.
\newblock {\em Ann. Inst. Fourier (Grenoble)}, 54(1):159--179, 2004.

\bibitem[Cha16]{Cha16}
Mohamad Charabati.
\newblock Modulus of continuity of solutions to complex {Hessian} equations.
\newblock {\em Internat. J. Math.}, 27(01):1650003, 2016.

\bibitem[CM21]{ChuMcCleerey}
Jianchun Chu and Nicholas McCleerey.
\newblock Fully non-linear degenerate elliptic equations in complex geometry.
\newblock {\em J. Funct. Anal.}, 281(9):109176, 2021.

\bibitem[CZ19]{ChuZhou2019}
Jiayu Chu and Ben Zhou.
\newblock Optimal regularity of plurisubharmonic envelopes on compact hermitian manifolds.
\newblock {\em Sci. China Math.}, 62(3):371--380, 2019.

\bibitem[Dar19]{Dar_survay}
Tamas Darvas.
\newblock Geometric pluripotential theory on {K}ähler manifolds.
\newblock {\em Advances in Complex Geometry}, 735:1--104, 2019.

\bibitem[DDG{\etalchar{+}}14]{DDG+14}
Jean-Pierre Demailly, Sławomir Dinew, Vincent Guedj, H.-H. Ph{\d{a}}m, Sławomir Kołodziej, and Ahmed Zeriahi.
\newblock Hölder continuous solutions to {Monge–Ampère} equations.
\newblock {\em J. Eur. Math. Soc. (JEMS)}, 16(4):619--647, 2014.

\bibitem[DDL21]{DDL21}
Tamás Darvas, Eleonora DiNezza, and Chinh~H. Lu.
\newblock Log-concavity of volume and complex {M}onge–{A}mpère equations with prescribed singularity.
\newblock {\em Math. Ann.}, 379:95–132, 02 2021.

\bibitem[DK14]{DK14}
S{\l}awomir Dinew and S{\l}awomir Ko{\l}odziej.
\newblock A priori estimates for complex {H}essian equations.
\newblock {\em Anal. PDE}, 7(1):227–244, 2014.

\bibitem[DP10]{DP2010}
Jean-Pierre Demailly and Nefton Pali.
\newblock Degenerate complex {Monge--Ampère} equations over compact kähler manifolds.
\newblock {\em Internat. J. Math.}, 21(3):357--405, 2010.

\bibitem[EGZ08]{EGZ08}
Philippe Eyssidieux, Vincent Guedj, and Ahmed Zeriahi.
\newblock A priori $\mathrm{L}^\infty$ estimates for degenerate complex {Monge–Ampère} equations.
\newblock {\em International Mathematics Research Notices}, 2008:rnn070, 01 2008.

\bibitem[G{\aa}r59]{gar59}
Lars G{\aa}rding.
\newblock An inequality for hyperbolic polynomials.
\newblock {\em Journal of Mathematics and Mechanics}, 8(6):957--965, 1959.

\bibitem[GGZ25]{GGZ25}
Vincent Guedj, Henri Guenancia, and Ahmed Zeriahi.
\newblock Diameter of kähler currents.
\newblock {\em J. Reine Angew. Math.}, 2025(820):115--152, 2025.

\bibitem[GL22]{GL2022}
Vincent Guedj and Chinh~H. Lu.
\newblock Quasi-plurisubharmonic envelopes 2: Bounds on {M}onge–-{A}mpère volumes.
\newblock {\em J. Algebraic Geom.}, page 688–713, 2022.

\bibitem[GL23]{GL21}
Vincent Guedj and Chinh~H. Lu.
\newblock Quasi-plurisubharmonic envelopes 3: Solving {M}onge–{A}mpère equations on hermitian manifolds.
\newblock {\em J. Reine Angew. Math.}, 2023(800):259--298, 2023.

\bibitem[GL24]{GL21a}
Vincent Guedj and Chinh~H. Lu.
\newblock Quasi-plurisubharmonic envelopes 1: uniform estimates on {K}ähler manifolds.
\newblock {\em J. Eur. Math. Soc. (JEMS)}, 27(3):1185–1208, 2024.

\bibitem[GL25a]{guedj2023degeneratecomplexhessianequations}
Vincent Guedj and Chinh~H. Lu.
\newblock Degenerate complex {Hessian} equations on compact {Hermitian} manifolds.
\newblock {\em Pure Appl. Math. Q.}, 21(3):1171--1194, 2025.

\bibitem[GL25b]{gl2025}
Vincent Guedj and Chinh~H. Lu.
\newblock Uniform estimates: from {Y}au to {K}o{\l}odziej, 2025.

\bibitem[GLZ19]{GLZ2019envelopes}
Vincent Guedj, Chinh~H. Lu, and Ahmed Zeriahi.
\newblock Plurisubharmonic envelopes and supersolutions.
\newblock {\em J. Differential Geom.}, 113(2):273--313, October 2019.

\bibitem[GN18]{GuNguyen2018}
Dongwei Gu and Ngoc~Cuong Nguyen.
\newblock The {D}irichlet problem for a complex {Hessian} equation on compact {Hermitian} manifolds with boundary.
\newblock {\em Ann. Sc. Norm. Super. Pisa Cl. Sci. (5)}, XVIII:1189--1248, 2018.

\bibitem[GP24]{GP24}
Bin Guo and Duong Phong.
\newblock On {$L^\infty$} estimates for fully non-linear partial differential equations.
\newblock {\em Annals of Mathematics}, 200:365--398, 07 2024.

\bibitem[GPT23]{GPT23}
Bin Guo, Duong Phong, and Freid Tong.
\newblock On {$L^\infty$} estimates for complex {Monge-Ampère }equations.
\newblock {\em Annals of Mathematics}, 198:393--418, 07 2023.

\bibitem[GPTW21]{guo2021moduluscontinuitysolutionscomplex}
Bin Guo, Duong~H. Phong, Freid Tong, and Chuwen Wang.
\newblock On the modulus of continuity of solutions to complex {Monge-Amp\`ere} equations, 2021.

\bibitem[KN15]{KN15weaksol}
Slawomir Ko{\l}odziej and Ngoc~Cuong Nguyen.
\newblock Weak solutions to the complex {Monge-Amp\`ere} equation on hermitian manifolds.
\newblock {\em Contemp. Math.}, 644:141--158, 06 2015.

\bibitem[KN16]{Ko_odziej_2016}
S{\l}awomir Ko{\l}odziej and Ngoc~Cuong Nguyen.
\newblock Weak solutions of complex {H}essian equations on compact hermitian manifolds.
\newblock {\em Compos. Math.}, 152(11):2221–2248, 2016.

\bibitem[KN19]{KN2019Stability}
Sławomir Kołodziej and Ngoc~Cuong Nguyen.
\newblock Stability and regularity of solutions of the {Monge–Ampère }equation on hermitian manifolds.
\newblock {\em Adv. Math.}, 346:264--304, 2019.

\bibitem[KN21]{KN2021continuous}
S{\l}awomir Ko{\l}odziej and Ngoc~Cuong Nguyen.
\newblock Continuous solutions to {M}onge--{A}mp{\`e}re equations on hermitian manifolds for measures dominated by capacity.
\newblock {\em Calc. Var. Partial Differential Equations}, 60(3):93, 2021.

\bibitem[KN23]{kolodziej2023complexhessianmeasuresrespect}
S{\l}awomir Ko{\l}odziej and Ngoc~Cuong Nguyen.
\newblock Complex {H}essian measures with respect to a background hermitian form.
\newblock {\em Anal. PDE, to appear}, page~~, 2023.
\newblock Preprint available at \url{https://arxiv.org/abs/2308.10405}.

\bibitem[Ko{\l}98]{Kolodziej1998}
S{\l}awomir Ko{\l}odziej.
\newblock The complex {M}onge-{A}mpère equation.
\newblock {\em Acta Math.}, 180(1):69--117, 1998.

\bibitem[Ngu14]{Nguyen_2014}
Ngoc~Cuong Nguyen.
\newblock Hölder continuous solutions to complex hessian equations.
\newblock {\em Potential Anal.}, 41(3):887–902, 04 2014.

\bibitem[Ngu16]{NG16}
Ngoc~Cuong Nguyen.
\newblock The complex {M}onge–{A}mpère type equation on compact hermitian manifolds and applications.
\newblock {\em Adv. Math.}, 286:240--285, 2016.

\bibitem[Qia24]{qiao2024sharpmathrmlinftyestimatesfully}
Yuxiang Qiao.
\newblock Sharp $\mathrm{L}^\infty$ estimates for fully non-linear elliptic equations on compact complex manifolds, 2024.

\bibitem[Sz{\'e}18]{Sze18}
G{\'a}bor Sz{\'e}kelyhidi.
\newblock {Fully non-linear elliptic equations on compact Hermitian manifolds}.
\newblock {\em J. Differential Geom.}, 109(2):337 -- 378, 2018.

\bibitem[Tos18]{Tosatti_2018}
Valentino Tosatti.
\newblock Regularity of envelopes in {Kä}hler classes.
\newblock {\em Math. Res. Lett.}, 25(1):281–289, 2018.

\bibitem[TW15]{Tosatti_2015}
Valentino Tosatti and Ben Weinkove.
\newblock On the evolution of a hermitian metric by its {Chern-Ricci} form.
\newblock {\em J. Differential Geom.}, 99(1):125 -- 163, 01 2015.

\bibitem[Yau78]{Yau78}
Shing-Tung Yau.
\newblock {On the {R}icci curvature of a compact {K}ähler manifold and the complex {Monge--Amp\`ere} equation, I}.
\newblock {\em Comm. Pure Appl. Math.}, 31(3):339--411, 1978.

\end{thebibliography}

\end{document}